\documentclass[10pt]{amsart}
\usepackage[english]{babel}
\usepackage{fancyhdr}
\usepackage{stmaryrd}
\usepackage{calrsfs}

\usepackage{amsmath}
\usepackage[nospace,noadjust]{cite}
\usepackage{amsfonts}
\usepackage{amssymb,enumerate}
\usepackage{amsthm}
\usepackage{cite}
\usepackage{comment}
\usepackage{color}
\usepackage[all]{xy}
\usepackage{hyperref}
\usepackage{lineno}
\usepackage{mathtools}
\usepackage{tikz-cd}
\usepackage{xy}
\input xy
\xyoption{all}
\usepackage{stmaryrd}
\usepackage{calrsfs}
\usepackage{caption}

\newtheorem{Th}{Theorem}[section]
\newtheorem{Lm}[Th]{Lemma}
\newtheorem{Pro}[Th]{Proposition}
\newtheorem{Co}[Th]{Corollary}
\newtheorem{Ep}[Th]{Example}
\newtheorem{Df}[Th]{Definition}
\newtheorem{rk}[Th]{Remark}
\newtheorem{Qs}[Th]{Question}
\theoremstyle{definition}
\newtheorem{definition}[Th]{Definition}

\numberwithin{equation}{section}

\newcommand{\F}{\mathbb{F}}
\newcommand{\N}{\mathbb{N}}
\newcommand{\Q}{\mathbb{Q}}
\newcommand{\R}{\mathbb{R}}
\newcommand{\Z}{\mathbb{Z}}

\DeclareMathOperator{\ord}{ord}

\providecommand\ldb{\llbracket}
\providecommand\rdb{\rrbracket}

\hypersetup{
	pdftitle={Sub-atomic},
	colorlinks=true,
	linkcolor=blue,
	citecolor=cyan,
	urlcolor=wine
}

\keywords{atomic domain, idf-domain, atomicity, sub-atomic domain}
\subjclass[2020]{Primary 13A05; Secondary 13F15, 13G05}

\begin{document}

	\title{A Weaker Notion of Atomicity in Integral Domains}
	
	\author{Mohamed Benelmekki}
	\address{Department of Mathematics, Facult\'e des Sciences et Techniques, 
Beni Mellal University, P.O. Box 523, Beni Mellal, Morocco}
	\email{med.benelmekki@gmail.com}
	
	\author{Brahim Boulayat}
	\address{Department of Mathematics and Computer Science, FPK, 
Sultan Moulay Slimane University, P.O. Box 145, Khouribga, Morocco}
	\email{boulayat.bra@gmail.com}
	
\date{\today}
\begin{abstract}
	A nonzero element of an integral domain $D$ is called \emph{atomic} if it is a unit or can be written as a finite product of irreducible elements of $D$, and $D$ is called \emph{atomic} if every nonzero element of $D$ is atomic. Here, we introduce and study a weaker notion of atomicity, which relaxes the requirement that all elements admit a factorization into irreducibles. Namely, we say that an integral domain is \emph{sub-atomic} if every divisor of an atomic element is also atomic. Then, we investigate the basic properties of such domains, provide examples, and explore the behavior of the sub-atomic property under standard constructions such as localization and $D+M$ constructions.  In particular, we prove that the sub-atomic  property ascends to polynomial rings in the class of GL-domains.   Our results highlight the independence of the sub-atomic property from other classical factorization properties, and it forms an important class of integral domains lying between atomic and antimatter domains (i.e., domains without atoms).
\end{abstract}
\medskip

\maketitle


\bigskip
\section{Introduction}

The classical theory of factorization in integral domains originates in the nineteenth-century work of Kummer, Dedekind, and Kronecker, whose efforts to restore unique factorization in algebraic number rings led to the creation of ideals and the foundational tools of modern commutative algebra. As the subject developed, attention naturally expanded from the uniqueness of factorization to the more basic question of the existence of factorizations. By the mid-twentieth century, the concepts of irreducible and prime elements had become standard. A nonzero  element of an integral domain is called \emph{atomic}  if it is a unit or can be written as a finite product of irreducible elements, and an integral domain is called \emph{atomic} if every nonzero element is atomic. Introduced by Cohn in \cite{C68}, the notion of an atomic domain emerged as the fundamental baseline in factorization theory. It is well known that every Noetherian domain is atomic. Indeed, the Noetherian property implies the ascending chain condition on principal ideals (ACCP), and any domain satisfying ACCP is atomic; see \cite{AAZ90} for further details.

 However, it soon became clear that atomicity is far too restrictive for many 
natural classes of non-Noetherian domains. Valuation domains with dense value groups, integer-valued polynomial rings, and an array of pullback constructions often contain nonunit elements that admit no factorization into irreducibles. Throughout the 1980s--1990s, work by Anderson, Anderson, Zafrullah, and others systematically 
demonstrated that large families of integral domains lie strictly below atomicity, 
yet still exhibit meaningful factorization-like behavior (see \cite{AAZ90,AAZ92}). This recognition led to the introduction of several weaker notions of atomicity, designed to measure how far a domain deviates from the classical atomic condition while still retaining some structural control via irreducibles.

Among these weaker notions, three concepts have become central: \emph{near-atomic}, \emph{almost-atomic}, and \emph{quasi-atomic} domains. These properties were introduced to capture progressively weaker factorization conditions in the non-atomic setting, and they now constitute an active area of research in factorization theory. The earliest of these notions is the \emph{almost-atomic} domain, introduced by Boynton and Coykendall. An integral domain $D$ is almost-atomic if every nonzero nonunit can be made atomic by multiplication with finitely many irreducible elements \cite{BC15}. The notion of a \emph{near-atomic} domain was introduced by Lebowitz-Lockard in \cite{L19}. The  domain $D$ is called near-atomic if there is a fixed element $0\neq \beta\in D$ such that, for any nonzero nonunit $a\in D$, the product $ a\beta$ is an atomic element of $D$. In this sense, atomic behavior does not need to occur at $a$ itself, but emerges after multiplying $a$ by $\beta$. Near-atomicity provides a useful bridge between atomic and almost-atomic. The weakest of the three notions, and in some sense the most conceptually important, is the \emph{quasi-atomic domain}, developed in work of Boynton and Coykendall. A domain is quasi-atomic if every nonunit divides some atomic element \cite{BC15}. These notions form the following strictly descending chain:
\smallskip
	\begin{center}
		\begin{tikzcd}
			\textbf{atomic} \arrow[r, Rightarrow, shift right=0.8ex] 
			\arrow[r, red, Leftarrow, "/"{anchor=center,sloped}, shift left=0.8ex] & 
			\textbf{near-atomic} \arrow[r, Rightarrow, shift right=0.8ex] 
			\arrow[r, red, Leftarrow, "/"{anchor=center,sloped}, shift left=0.8ex] & 
			\textbf{almost-atomic} \arrow[r, Rightarrow, shift right=0.8ex] 
			\arrow[r, red, Leftarrow, "/"{anchor=center,sloped}, shift left=0.8ex] & 
			\textbf{quasi-atomic},
		\end{tikzcd}
	\end{center}
where none of the implications can be reversed. A significant study addressing these three notions of atomicity is presented in \cite{L19}.
\bigskip

These weaker conditions do not require that any specific element factor into irreducibles, nor that such factorizations be unique; rather, they only require that every nonunit divides at least one atomic element of the domain. However, in a domain that is not a field, all of these conditions, including atomicity, require the existence of at least one atom (irreducible element). In this paper, we propose a weaker notion of atomicity that does not require the existence of atoms, yet atomic behavior can still appear in the divisibility structure even when full atomicity is absent.
\begin{Df}  We say that an integral domain $D$ is a \emph{sub-atomic domain} if every  divisor of an atomic element of $D$ is also atomic in $D$. 
\end{Df}

This definition provides a natural generalization of atomic domains, defining a class of integral domains that lies between atomic and antimatter domains (i.e., domains without atoms). An important subclass of sub-atomic domains consists of AP-domains (domains in which every atom is prime), see Proposition~\ref{NewPro}.  Since an antimatter domain has no atoms, it is vacuously an AP-domain, which situates these classes of domains as follows:

\begin{center}
	\begin{tikzcd}[column sep=large]
		\textbf{antimatter}
		\arrow[r, Rightarrow, shift right=0.8ex]
		\arrow[r, red, Leftarrow, "/"{anchor=center,sloped}, shift left=0.8ex] &
		\textbf{AP}
		\arrow[r, Rightarrow, shift right=0.8ex]
		\arrow[r, red, Leftarrow, "/"{anchor=center,sloped}, shift left=0.8ex] &
		\textbf{sub-atomic}
		\arrow[r, Leftarrow, shift right=0.8ex]
		\arrow[r, red, Rightarrow, "/"{anchor=center,sloped}, shift left=0.8ex] &
		\textbf{atomic}
	\end{tikzcd}
\end{center}\smallskip

 Thus, unlike classical atomic domains, sub-atomic domains do not require every nonunit to factor into irreducibles, which enables the study of factorization behavior in a broader context. At the same time, they allow the transfer of atomic behavior from elements to their divisors. On the other hand, the usual definitions of factorization properties in integral domains (such as UFD, HFD, FFD, and BFD) always assume that the domain is atomic \cite{AAZ90}. A natural question then arises: what happens if this assumption is removed or replaced by a weaker condition such as sub-atomic? In this context, Coykendall and Zafrullah introduced the notion of an unrestricted unique factorization domain (U-UFD), referring to domains in which every atomic element admits a unique factorization \cite{CZ}. Every UFD is therefore a U-UFD, but the converse does not hold. An AP-domain is also a U-UFD; see \cite[Lemma 2.2]{CZ}.  More recently, in 2025, Du and Gotti introduced a related concept that does not require the domain to be atomic. They defined an integral domain to be an \emph{unrestricted finite factorization domain} (U-FFD) if every atomic element admits only finitely many nonassociate factorizations \cite{DG}.\bigskip

The goal of this paper is to study the sub-atomic concept, which does not require the domain to be atomic. Along the way, we introduce an interesting generalization of IDF-domains, which we call \emph{restricted-irreducible-divisor-finite} domains (RIDF-domains). Let $D$ be an integral domain and let $D^* = D \setminus \{0\}$. The domain $D$ is called an \emph{IDF-domain} if every nonzero element of $D$ has only finitely many nonassociate irreducible divisors \cite{GW75}. We say that $D$ is an \emph{RIDF-domain} if every atomic element of $D$ has only finitely many nonassociate irreducible divisors in $D$. Then, we investigate the basic properties of such domains, provide examples, and explore the behavior of the sub-atomic property under standard ring constructions such as localization and $D+M$ constructions. In particular, we prove that the sub-atomic (resp., RIDF, U-FFD) property ascends to polynomial rings in the class of GL-domains (resp., PSP-domains). \bigskip

 The remainder of the paper is organized as follows. In Section~\ref{sec2}, we recall the necessary background and set up the notation needed for our study. We review several definitions and results on factorization properties in monoids and integral domains, and we present monoid domains as a central tool for constructing examples and counterexamples. In Section~\ref{sec3}, we introduce sub-atomic domains and RIDF-domains, and characterize the U-FF property for sub-atomic domains. We also investigate how the sub-atomic (resp., RIDF) property behaves under localization and the $D+M$ construction. Section~\ref{sec4} is devoted to the study of the ascent of the sub-atomic (resp., RIDF, U-FF) property to polynomial extensions.

	\section{Notation and Background} \label{sec2}
			In this section we recall some terminology and fix the notation that we need in this paper.\smallskip
			
			We denote by $\N$, $\Z$, $\Q$, and $\R$ the sets of positive integers, integers, rational numbers, and real numbers, respectively. We let $\Z_+$ and $\Q_+$ denote the sets of nonnegative integers and rational numbers, respectively. Throughout this paper, all monoids are assumed to be commutative cancellative semigroups with identity.\smallskip

Let $M$ be a monoid, written multiplicatively. An element $b \in M$ is a \emph{unit} if there exists $c \in M$ such that $b c = 1$. We let $U(M)$ denote the group of unit elements of $M$, and we say that $M$ is reduced if $U(M)$ is trivial. When $M$ is the multiplicative monoid of an integral domain $D$, the unit elements of $M$ are precisely the units of $D$. A nonempty subset $N\subseteq  M$ is a submonoid of $M$ if it is closed under multiplication and contains the identity element of $M$. Let $S\subseteq  M$. The submonoid of $M$ generated by $S$, denoted by $\langle S\rangle$, is the intersection of all the submonoids of $M$ containing $S$. \smallskip
			
			Let $a,b\in M$. We say that $b$ is divisible by $a$, or that $a$ divides $b$, if there exists $c\in M$ such that $b=ac$; in this case we write $a\mid_M b$ (or simply $a \mid b$ when the context is clear). We say that $a$ and $b$ are associate if $a\mid_M b$ and $b\mid_M a$; we write $a\sim b$. A submonoid $N$ of $M$ is said to be \emph{divisor-closed} if for each $a\in N$, every divisor of $a$ in $M$ lies in $N$. A nonunit element $a\in M$ is called an atom (or irreducible element) if $a=bc$ for some $b,c\in M$ implies that $b\in U(M)$ or $c\in U(M)$. We denote the set of atoms of $M$ by $\mathcal{A}(M)$. \smallskip
			
			Following Cohn \cite{C68}, we say that a monoid  $M$ is \emph{atomic} if every nonunit $a\in M$ can be expressed as a finite product of atoms in $M$; i.e.,  $a=a_1a_2\cdots a_n$, where each $a_i\in M$ is irreducible. An atomic monoid $M$ is called a \emph{BFM} if for every nonzero nonunit element $a \in M$ there exists 
$n_0 \in \mathbb{N}$ such that $n \leqslant n_0$ for every atomic factorization of $a$ in $M$; that is, the 
length of any atomic factorization of $a$ in $M$ is bounded by $n_0$. An element of $M$ is called \emph{atomic} if it is a unit or can be written as a finite product of atoms. The set of atomic elements of $M$ will be denoted by $A(M)$. Note that $A(M)$ is a submonoid of $M$, and $M$ is an atomic monoid when $A(M)=M$. \smallskip
			
A monoid $M$ is called an IDF-monoid (or is said to have the IDF-property) if every nonzero element of $M$ has only finitely many nonassociate irreducible divisors. Note that a monoid with no atoms (i.e., an antimatter monoid) is vacuously an IDF-monoid. We say that $M$ is a \emph{finite factorization monoid} (FFM) if $M$ is atomic and every nonunit element of $M$ has only finitely many nonassociate factorizations in $M$. \smallskip
			
Recently, several researchers have studied these classical factorization properties without assuming that the base monoid is atomic. We give the following definition:
			\begin{Df}
			
				 A monoid $M$ is called an unrestricted finite factorization monoid (U-FFM) or has the U-FF property if every atomic element of $M$ has only finitely many nonassociate factorizations \cite{DG}.

			\end{Df}
			
			Let $D$ be an integral domain and let $D^*=D\backslash \{0\}$. Note that $D^*$ is a monoid under multiplication, called the multiplicative monoid of $D$.
			An integral domain $D$ is called atomic (resp., IDF, FFD, BFD, U-FFD)  if its multiplicative monoid $D^*$ is an atomic (resp., IDF, FF, BF, U-FF) monoid. 
				
\begin{Df} We say that a monoid $M$ is a \emph{sub-atomic monoid}  if every divisor of an atomic element of $M$ is atomic (i.e., for every atomic  element $a$ of $M$, if $a=bc$ for some $b \in M$ and $c\in M$, then  $b$ is atomic). We say that an integral domain $D$ is a \emph{sub-atomic domain}  if its multiplicative monoid is sub-atomic.
			\end{Df}
		It is clear that an atomic domain/monoid is sub-atomic. \smallskip
			
Let $D$ be an integral domain. Let $f \in D[X]$ be a nonzero polynomial, and let $A_f$ be the ideal of $D$
generated by the coefficients of $f$. We say that $f$ is \emph{primitive}
(resp., \emph{super-primitive}) if the only common divisors of its coefficients
are the units of $D$ (resp., $A_f^{-1}=D$). Following \cite{AS75}, we say that:
			\begin{itemize}
				\item[$\bullet$] $D$ is a GL-domain if the product of two primitive polynomials in $D$ remains a primitive polynomial.
				\item[$\bullet$] $D$ is a PSP-domain if every primitive polynomial over $D$ is super-primitive.
			\end{itemize}	
			Let $S$ be a nonempty subset of $D$. Recall that a common divisor $d\in D$ of $S$ is called a \emph{maximal common divisor} (MCD) if for every common divisor $c\in D$ of $S$, if $d\mid c$ then $c\sim d$. The domain $D$ is called an MCD-domain if every nonempty subset of $D$ has at least one maximal common divisor \cite{Rot}. Following \cite{EK}, we say that $D$ is an  MCD-finite domain if every nonempty finite subset of nonzero elements of $D$ has only a finite number of maximal common divisors (possibly zero).\smallskip

			We end this section by recalling the definition of the main tool used in this article.
			\begin{definition} Let $D$ be an integral domain, and let $X$ be an indeterminate over $D$. Let $M$ be a cancellative torsion-free monoid. The monoid domain of $M$ over $D$, denoted by $D[M]$, consists of all polynomial expressions with coefficients in $D$ and exponents in $M$:
				
					$$D[M]=\lbrace\sum_{i=1}^{n} a_iX^{s_i}; \;\; (a_i,s_i)\in D\times M\;\; \textit{for every}\;\; i=1,\ldots,n\rbrace .$$
					\end{definition}
Since the monoid $M$ is cancellative and torsion-free, it embeds in its quotient  group, which is a torsion-free commutative group. By a result of Levi~\cite{L13}, such a group admits a total order compatible with the group operation; hence $M$ can be equipped with a total ordering $\prec$ compatible with its monoid structure.
   Then every nonzero element $f\in D[M]$ is uniquely expressible as:
				$$f=a_1X^{s_1}+\cdots+a_n X^{s_n},$$	where $a_i \neq 0$ for every $i=1,\ldots,n$,  and $s_1\prec s_2\prec  \cdots\prec s_n$. In this case, $s_1$  is called the order of $f$ and is denoted by $\ord f$. The element $s_n$ (resp., $a_n$) is called the degree (resp., the leading coefficient) of $f$ and is denoted by $\deg f$ (resp., $lc(f)$). On the other hand,  $D[M]$ is an integral domain (commutative with identity) such that	$$U(D[M])=\{dX^{s},\;\; \textit{where}\;\; d\in U(D)\;\; and \;\; s\in U(M)\},$$
				see \cite[Corollary~4.2]{GP74}. We note that $D[M]$ is a GCD-domain if and only if $D$ is a GCD-domain and $M$ is a GCD-monoid \cite{GP74}. 
Monoid domains generalize polynomial rings. In particular,  $D[X]=D[M]$ for $M=\Z_+$. These domains have been studied extensively and have proven useful for constructing 
examples and counterexamples in various areas of commutative ring theory \cite{G72,GP74,K01}. 
			 \\

\section{Sub-atomic RIDF-domains}\label{sec3}

In this section, we investigate the sub-atomic and RIDF properties. We first introduce these notions and examine their relationship with the U-FF property. 
\smallskip

Recall that an integral domain $D$ is called a sub-atomic domain if every nonunit divisor of an atomic element of $D$ is atomic in $D$.  It is immediate that every atomic domain (resp., monoid) is sub-atomic. However, an antimatter domain is, by definition, sub-atomic, yet it is not atomic unless it is a field. For other examples of sub-atomic domains that fail to be atomic, see Examples~\ref{EP1}, \ref{EP2}, \ref{EP3}, and \ref{Ep4}. On the other hand, we have the following straightforward result.

\begin{Pro}\label{NewPro}
Let $D$ be an integral domain and let $x = u a_1 a_2 \cdots a_n$ for some primes $a_1, \dots, a_n \in D$ and $u \in U(D)$. Then every nonunit divisor of $x$ also has a prime factorization in $D$.

In particular, every AP-domain is a sub-atomic domain.
\end{Pro}

\begin{proof}
Let $y \in D$ be a nonunit divisor of $x$, say $x = yz$ for some $z \in D$, so that $u a_1 a_2 \cdots a_n = yz$. We proceed by induction on $n \ge 1$.

For $n = 1$, since $x = u a_1$ is a prime element of $D$, the result is obvious.

For $n > 1$, since the prime $a_n$ divides $yz$, we consider two cases:
\begin{itemize}
    \item \textbf{Case 1:} If $a_n \mid y$, write $y = a_n y'$. Then $u a_1 \cdots a_{n-1} = y' z$. If $y' \notin U(D)$, it has a prime factorization by the induction hypothesis, so multiplying by $a_n$ gives a prime factorization of $y$. Otherwise, if $y' \in U(D)$, then $y = a_n y'$ is a prime factorization of $y$.
    \item \textbf{Case 2:} If $a_n \mid z$, write $z = a_n z'$. Then $u a_1 \cdots a_{n-1} = y z'$. By the induction hypothesis, $y$ has a prime factorization.
\end{itemize}

Finally, if $D$ is an AP-domain and $x \in D \setminus U(D)$ is atomic, then $x$ is a finite product of atoms (primes). Every nonunit divisor of $x$ therefore admits a prime factorization, and thus a factorization into atoms, so $D$ is a sub-atomic domain.
\end{proof}

Although sub-atomic domains are not atomic in general, atomicity can be characterized in terms of sub-atomicity and quasi-atomicity.  The following proposition makes this precise.  Before stating it, recall that an integral domain $D$ is quasi-atomic if for every nonzero nonunit $b\in D$, there exists an element $a\in D$ such that $ab$ is atomic in $D$.
\begin{Pro} Let $D$ be an integral domain. Then $D$ is atomic if and only if $D$ is both  sub-atomic 	and quasi-atomic.
\end{Pro}
\begin{proof}
The direct implication is obvious. For the reverse implication, assume that $D$ is both sub-atomic and quasi-atomic. Let $0\neq x\in D$. Since $D$ is quasi-atomic, there exists an element $a\in D$ such that $xa$ is atomic. Now $x$ is a divisor of the atomic element $xa$. By the sub-atomic property, $x$ must be atomic. Hence every nonzero element of $D$ is atomic, which means $D$ is an atomic domain.
\end{proof}
Next we introduce a weaker notion of the IDF-property.

\begin{Df}
    We say that a monoid $M$ is a \emph{restricted-irreducible-divisor-finite monoid} (RIDF-monoid)  if every  atomic element of $M$ has only a finite number of nonassociate irreducible divisors. We say that an integral domain is a \emph{restricted-irreducible-divisor-finite domain} (RIDF-domain)  if its multiplicative monoid is a RIDF-monoid. \end{Df}
An IDF-domain is obviously a RIDF-domain. However, the converse is not true in general, as the following example shows. 
\begin{Ep} \label{EP1} Consider a GCD-domain $D$ which is not an IDF-domain. For instance,  fix a prime number $p$ and let $G$ be the subgroup of the additive group of rational numbers $\Q$ generated by the following set $\lbrace \frac{1}{p^k}\; |\;k\in \mathbb{N}\rbrace$. Then, it follows from \cite[Theorem~5.2]{GP74} that the group algebra $D:=\Q[G]$ is a GCD-domain. On the other hand, $D$ is not an IDF-domain by \cite[Example~3.8]{BEJA}.\smallskip

Now, let us prove that $D$ is a RIDF-domain. Let $f$ be an atomic element of $D$, say $f=u\prod_{i=1}^n f_i^{e_i}$ where $u\in U(D)$, $e_i\in \N$, and the $f_i$'s are nonassociate irreducible elements of $D$. Since $D$ is a GCD-domain, each $f_i$ is  prime in $D$. Therefore, the elements $f_1,\ldots, f_n$ are the only nonassociate irreducible (prime) divisors of $f$ in $D$, and hence $f$ satisfies the RIDF-property. Consequently, $D$ is a RIDF-domain which is not an IDF-domain. Moreover, $D$ is a sub-atomic domain. Indeed, every atomic element has a prime factorization in $D$, and by Proposition~\ref{NewPro}, every nonunit divisor of such an element also admits a prime factorization (hence is atomic). However, $D$ is not atomic because $X - 1$ is not atomic in $D$.
\end{Ep}

In general, the RIDF-property does not imply  the sub-atomic property, as the following example shows (see also Example \ref{EP3}). 
  
\begin{Ep} \label{g}
	Let $p$ and $q$ be two prime numbers such that $p\neq q$. Consider the following Puiseux monoid
	\[
		M := \bigg\langle \frac {1}{q},  \frac 1{p^k} \ \Big{|} \ k \in \N \bigg\rangle.
	\]
	One can easily check that $\mathcal{A}(M) = \big\{ \frac{1}{q} \big\}$. Therefore $M$ is an IDF-monoid, and hence it is an RIDF-monoid. But $M$ is not a sub-atomic monoid. To see this, observe that $ 1=q \times\frac{1}{q}$,  and then $1$ has an atomic factorization in $M$. On the other hand, for every $k\in  \N$ we have $1=p^k\times\frac 1 {p^k}$, and so $\frac 1{p^k}$ is a divisor of $1$ in $M$. However, $\frac 1 {p^k}$ has no irreducible divisor in $M$ for every $k\in  \N$. In particular, $\frac 1{p^k}$ is not atomic in $M$ for every $k\in  \N$. Therefore, $M$ is not a sub-atomic monoid. 
\end{Ep}

Note that there exist integral domains that fail to be sub-atomic and are not RIDF-domains either. The following example illustrates this observation.
\begin{Ep}
	Let $X$ be an indeterminate over $\Q$. Consider the integral domain
	$$
	D := \Q \bigg[ \big\{  X^{2},X^{\frac{3}{2^k}}  : k \in \N \big\} \bigg].
	$$ 
Note that $D$ is isomorphic to the Puiseux algebra $\Q[M]$ where	$M := \langle 2,  \frac 3{2^k}  \ | \ k \in \N \rangle$. We claim that $D$ is not a sub-atomic domain. To see this, observe that  $\mathcal{A}(M)=\{2\}$ and then $X^2$ is an atom of $D$.  Since $X^6=\left( X^2\right)^3$, it follows that $X^6$ is atomic in $D$. However, $X^3$ is a divisor of $X^6$ in $D$ which is not atomic. Indeed,  since $1\notin M$ and $ \frac{3}{2^k}< 2$ for every $k\in \N$, $\gcd(X^2,X^3)=1$ (i.e., $X^2$ and $X^3$ share only units as common divisors in $D$) and hence $X^2$ cannot divide $X^3$ in $D$.    Moreover, $X^3=\left( X^{\frac{3}{2^k}}\right)^{2^k}$ such that $X^{\frac{3}{2^k}}$ has no irreducible divisor in $D$, for every $k\in \N$. Hence $D$ is not a sub-atomic domain.\smallskip
	
	On the other hand, $D$ is  not a RIDF-domain. To verify this,  let us first prove that the element $f:=X^6-1$ is atomic in $D$. We have $$f=\left(X^2-1\right)\left(X^4+X^2+1\right).$$
	Since $2$ is an atom of $M$, $X^2-1$ is irreducible in $D$. Note that $4$ has a unique factorization in $M$. 
Note that $4$ has a unique representation as a sum of generators in $M$. Indeed, suppose
$4 = 2n + \sum_{i=1}^{r} \frac{3}{2^{k_i}}
$ for some $n, r \in \mathbb{Z}_{+}$ and $k_1, \dots, k_r \in \mathbb{N}$. Set $e := \max\{k_1, \dots, k_r\}$ (with $e := 0$ if $r = 0$) and $m := \sum_{i=1}^{r} 2^{e - k_i} \in \mathbb{Z}_{+}$. We obtain
$2^{e+2} = 2^{e+1}n + 3m$, so $2^{e+1}(2 - n) = 3m$. Since $n, m \in \mathbb{Z}_{+}$, we must have $n \le 2$. If $n \in \{0, 1\}$, then $3$ would divide $2^{e+1}(2 - n)$, which is impossible. Thus $n = 2$, and hence $m = 0$. Consequently, $4 = 2 + 2$ is the unique expression of $4$ as a sum of atoms in $M$, which implies that every factorization of $X^4 + X^2 + 1$ in $D$ has length at most $2$.	 In particular, $X^4+X^2+1$ is atomic in $D$, and thus $f$ is atomic in $D$.  Let us now show  that $f$  has infinitely many nonassociate irreducible divisors in $D$. Note that for every $n\in \Z_+$ we have $$f=\left(X^3-1\right)\left(X^3+1\right)=\left(X^{\frac{3}{2^n}}-1\right)\prod_{k=0}^n\left(X^{\frac{3}{2^k}}+1\right).$$
	To conclude, we show that $\{X^{\frac{3}{2^n}}+1\}_{n\in \N}$ is an infinite set of nonassociate irreducible divisors of $f$ in $D$. Fix $n\in \N$. As previously denoted, $2$ does not divide $\frac{3}{2^n}$. Hence every divisor of $X^{\frac{3}{2^n}}+1$, in $R$, lies in $\Q[M']$ where $M' := \langle   \frac 3{2^k}  \ | \ k \in \N \rangle$. Let $G$ be the quotient group of $M'$, and set $h(T):=T+1\in \Q[T]$. Since $h(T^{2^d})=T^{2^d}+1=\Phi_{2^{d+1}}(T)$ is irreducible (the $2^{d+1}$-th cyclotomic polynomial) over $\Q$ for every $d\in \N$, it follows from \cite[Proposition 1]{BEGraz} that $h(X^{\frac{3}{2^n}})$ is irreducible in $\Q[G]$. Hence $X^{\frac{3}{2^n}}+1$ is irreducible in $\Q[G]$, and therefore it is irreducible in $D$. Consequently, $D$ is not a RIDF-domain. 
\end{Ep}

 
  The next result characterizes the U-FF property for sub-atomic domains. 
\begin{Pro}\label{P1dd}
Let $D$ be a sub-atomic domain. Then $D$ is a U-FFD if and only if it is an RIDF-domain.
\end{Pro}
\begin{proof}
Suppose that $D$ is a U-FFD. Let $x \in D \setminus U(D)$ be an atomic element, and let $\{x_i\}_{i \in I}$ be the set of all nonassociate irreducible divisors of $x$ in $D$. Then, for every $i \in I$, there exists $y_i \in D$ such that $x = x_i y_i$. Since $D$ is sub-atomic, $y_i$ is atomic in $D$.  Thus,  $x_i, y_i \in A(D^*)$, where $A(D^*)$ denotes the atomic submonoid of $D^*$. Hence, each $x_i$ is an irreducible divisor of $x$ in $A(D^*)$. By \cite[Proposition~3.6]{DG}, $A(D^*)$ is an FFM, and hence an IDF-monoid by \cite[Theorem~2]{H92}. Consequently, $I$ is finite, so $x$ satisfies the IDF-property in $D$. Therefore, $D$ is an RIDF-domain.

Conversely, assume that $D$ is an RIDF-domain. By \cite[Theorem~2]{H92}, the atomic submonoid $A(D^*)$ of $D^*$ is an FFM. Then \cite[Proposition~3.6]{DG} implies that $D^*$ is a U-FFM, and hence $D$ is a U-FFD.
\end{proof}
\begin{rk}
	We note that, in the previous proposition, the assumption that $D$ is sub-atomic is not necessary for the second implication. Consequently, the RIDF property implies the U-FF property. However, the converse does not hold in general. To illustrate this, we consider the same example studied by the authors in \cite[Example~3.8]{DG}. Let $(p_n)_{n\geqslant 1}$ be the strictly  increasing sequence of prime numbers, and 
		\[
	M := \bigg\langle \frac{p_n+1}{p_n^2} \ \Big{|} \ n \in \N \bigg\rangle \cup \mathbb{Q}_{>1}.
	\]
	We have that $M$ satisfies the U-FF property and that $\mathcal{A}(M)=\{a_n:  n\in \mathbb{N}\}$, where $a_n:=\frac{p_n+1}{p_n^2}$.\\
	To show that $M$ fails to be an $RIDF$-monoid, observe that $p_n+1=p_n^2 a_n$ in $M$ for all $n\in \N$. Hence $p_n+1$ is atomic for all $n\in \N$. However,  for every $n\in \N\setminus\{1\}$,  $p_n+1$ has infinitely many pairwise nonassociate irreducible divisors, since $2\mid_M (p_n+1)$ and $a_k \mid_M 2$ for all $k\in \mathbb{N}$. Therefore, 
		\begin{center}
		\begin{tikzcd}
			\textbf{IDF} \arrow[r, Rightarrow, shift right=0.8ex] 
			\arrow[r, red, Leftarrow, "/"{anchor=center,sloped}, shift left=0.8ex] & 
			\textbf{RIDF} \arrow[r, Rightarrow, shift right=0.8ex] 
			\arrow[r, red, Leftarrow, "/"{anchor=center,sloped}, shift left=0.8ex] & 
			\textbf{U-FF}
		\end{tikzcd}
	\end{center}
\end{rk}\bigskip

Next, we consider the sub-atomic property under the $D+M$ construction. Let $T$ be an integral domain that can be written in the form $T = K + M$, where  $K$ is a subfield of $T$ and $M$ is a nonzero maximal ideal of $T$. Let $D$ be a subring of $K$ and $R = D + M$. We now investigate how the sub-atomic property behaves under the $D+M$ construction.

\begin{Th} \label{Th1}
	Let $T$ be an integral domain of the form $K+M$, where  $K$ is a subfield of $T$ and $M$ is a nonzero maximal ideal of $T$.  Let $D$ be a subring of $K$ and $R = D + M$. Then the following statements hold.
\begin{enumerate}
\item Assume that $D$ is a field.  Then  $R$ is a sub-atomic domain if and only if $T$ is a sub-atomic domain.
\smallskip

\item Assume that $D$ is not a field. If $D$   and $T$ are sub-atomic, then  $R$ is a sub-atomic domain.
\end{enumerate} 
\end{Th}
\begin{proof} 
(1) Note that each element of $R$ (resp., $T$) is associated in $R$ (resp., $T$) to an element of the form $m$ or $1+m$ for some $m\in M$. Moreover, one can easily check, by \cite[Lemma~1.5(2)]{CMZ},  that each of these elements is irreducible in $R$ if and only if it is irreducible in $T$. Then the result follows immediately. 
\smallskip

 (2) Assume that $D$ and $T$ are sub-atomic, and let $a\in R\setminus U(R)$ be an atomic element of $R$, say $a = a_1 \cdots a_n$, where each $a_i \in R$ is irreducible. Since $D$ is not a field, no element of $M$ is irreducible in $R$. Indeed,  for any nonzero nonunit $d_0 \in D$ and any $m \in M$, we can write $m = d_0(d_0^{-1}m)$ with $d_0^{-1}m \in M$, expressing $m$ as a product of two nonunits of $R$. Consequently, no element of $M$ is atomic in $R$, and each irreducible factor $a_i$ satisfies $a_i \notin M$. Thus, for each $i = 1, \ldots, n$, we may write $a_i = d_i(1+m_i)$ for some $0 \neq d_i \in D$ and $m_i \in M$. If $d_i \notin U(R)$, then $1+m_i \in U(R)$, so $a_i$ is associate in $R$ to $d_i$, which forces $d_i$ to be irreducible in $R$ and hence in $D$; if $d_i \in U(R)$, then $d_i \in U(D)$. Therefore, $d := d_1 \cdots d_n$ is atomic in $D$. Setting $1+m := \prod_{i=1}^n (1+m_i)$ for some $m \in M$, we obtain $a = d(1+m)$. By \cite[Lemma~1.5(2)]{CMZ}, for every $\mu \in M$, $1+\mu$ is irreducible in $R$ if and only if it is irreducible in $T$, so $1+m$ is atomic in $T$. In particular, $a$ is atomic in $T$, since $d\in U(T)$.

Now let $b$ be a nonunit divisor of $a$ in $R$. Since $a=d(1+m)$ with $0\neq d\in D$, we have $a\notin M$. Hence $b\notin M$. Write $b=d'+m'$ with $d'\in D$ and $m'\in M$. Since $D$ is sub-atomic, $d'$ is atomic in $D$, and hence in $R$. To conclude that $b$ is atomic in $R$, it remains to show that $1+m'$ is atomic in $R$. Suppose that $1+m' \notin U(R)$. Then $1+m'$ is a nonunit divisor of $a$ in $T$. Since $T$ is sub-atomic and $a$ is atomic in $T$, $1+m'$ is atomic in $T$, say $1+m' = x_1 \cdots x_l$, where each $x_j \in T$ is irreducible. Each $x_j$ is of the form $k_j(1+\mu_j)$ for some $0 \neq k_j \in K$ and $\mu_j \in M$. Thus, $\prod_{j=1}^l k_j = 1$, and so $1+m' = \prod_{j=1}^l (1+\mu_j)$. Furthermore, \cite[Lemma~1.5(2)]{CMZ} implies that each $1+\mu_j$ is irreducible in $R$ for every $j = 1, \ldots, l$. Therefore, $1+m'$ is atomic in $R$, and hence $b = d'(1+m')$ is atomic in $R$. Consequently, $R$ is a sub-atomic domain. 
\end{proof}
\begin{rk}
Under the assumption of Theorem~\ref{Th1}, assume that $D$ is not a field. Note that the only possible atomic elements of $R$ are those of the form $d + m$, where $0 \neq d \in D$ and $m \in M$. Indeed, if some $m' \in M$ were atomic in $R$, then it would have at least one irreducible divisor $m_1 \in M$, which is impossible since $M$ contains no atoms of $R$. In particular, we cannot use an argument similar to that used in the proof of Theorem~\ref{Th1} to show that the sub-atomic property of $R$ implies that of $T$. However, if $R$ is a sub-atomic domain, then $D$ is also a sub-atomic domain. Indeed, let $d$ be an atomic element of $D$ and let $d_1 \notin U(D)$ such that $d=d_1d_2$ for some $d_2\in D$. Since $\mathcal{A}(D)\subseteq \mathcal{A}(R)$, $d$ is atomic in $R$. Since $R$ is sub-atomic and $d_1 \notin U(R)$ , $d_1$ is atomic in $R$, say $d_1=\prod_{i=1}^n (r_i+m_i)$, where each $r_i+m_i\in R$ is irreducible. Thus $d_1=\prod_{i=1}^n r_i$. Note that each $r_i$ is irreducible in $D$, since otherwise $r_i=u_iv_i$ for some $u_i,v_i \notin U(D)$ with $r_i+m_i=u_i(v_i+u_i^{-1}m_i)$, which is impossible because $u_i,(v_i+u_i^{-1}m_i)\notin U(R)$. Therefore, $d_1$ is atomic in $D$, and hence $D$ is a sub-atomic domain.

\end{rk}

\begin{Co}\label{c1}
 Let $D$ be an integral domain with quotient field $K$, and let $L$ be a field extension of $K$. Let $R = D + X L[X]$ and $R' = D + X L \ldb X \rdb$.
 \begin{enumerate}
\item Assume that $D$ is a field.  Then  $R$ and $R'$ are  sub-atomic domains.
\smallskip

\item Assume that $D$ is not a field. If $D$ is sub-atomic, then  $R$ and $R'$ are  sub-atomic domains.
\end{enumerate}
\end{Co}
\begin{proof} This follows from Theorem~\ref{Th1} since  $T=L[X]$ (resp., $T'=L \ldb X \rdb$) is a UFD, and hence it is  a sub-atomic domain. 
\end{proof}

\begin{Ep}\label{EP2}
Consider the integral domain $R = \Z + X\Q[X]$. Then, it follows from Corollary~\ref{c1}(2) that $R$ is a sub-atomic domain. However, $R$ is not atomic since $\Z$ is not a field; see \cite[Proposition~1.2(a)]{AAZ90}. For more details, we note that the only atomic elements of $R$ are those elements $f\in R\setminus \{-1,1\}$ such that $f(0)\neq 0$. Thus, all such elements, and therefore their nonunit divisors, are atomic in $R$ by the degree argument and the fact that $\Z$ is atomic.  On the other hand, for every $f\in R$ such that $f(0)= 0$,  $f=p(p^{-1}f)$ with $p^{-1}f\in R\setminus U(R)$ for every prime number $p$. In particular, every element  $f$ of this form has a factor divisible by every prime number, and so $R$ cannot be atomic. 
\end{Ep}

We conclude this section by examining the behavior of the sub-atomic property under localization. A saturated multiplicative subset $S$ of an integral domain $D$ is called \emph{splitting} if each $x \in D$ can be written as $x = as$ for some $a \in D$ and $s \in S$ such that $aD \cap tD = a tD$ for all $t \in S$. Note that the extension $D \subseteq D_S$ is inert when $S$ is a splitting multiplicative set; see \cite[Proposition~1.5]{AAZ92}. Also, note that $U(D_S)=S^{-1}S$.

Let $D$ be an  integral domain and let $\Gamma \subseteq D^*$ be a subset. The multiplicative set $S$ generated by $\Gamma$ is defined as follows:
$$S = \{u t_1 t_2 \dots t_n \mid u \in U(D), \ t_i \in \Gamma, \ n \in \Z_+ \}.$$
Note that,  since $S$ is closed under the associative multiplication of $D$ and contains the multiplicative identity, the set $S$ with the operation of multiplication from $D$ is a multiplicative submonoid of the multiplicative monoid $D^*$ of $D$.

\begin{Pro} \label{loc}
	Let $D$ be an integral domain, and let $S$ be a splitting multiplicative subset of $D$ generated by primes. Then $D$ is a sub-atomic domain if and only if $D_S$ is a sub-atomic domain.
\end{Pro}
\begin{proof}
 Assume that $D$ is a sub-atomic domain. Let $x$ be a nonunit atomic element of $D_S$, say $x=x_1\cdots x_n$, where each $x_i\in D_S$ is irreducible. We may suppose that $x\in D$ such that $xD \cap sD = xsD$ for every
  $s\in S$. In particular, no nonunit of $S$ divides $x$ in $D$. Thus, we may assume that each $x_i\in D$ and is irreducible in $D$. Hence, $x$ is atomic in $D$. 
 Now, let $y$ be a nonunit divisor of $x$, say $x=yz$ for some $z\in D_S$. Then, since $D \subseteq D_S$ is inert,  there exists $u\in U(D_S)$ such that $uy,u^{-1}z\in D$.  Note that $x$ and $uy$  are nonunits in $D$. Since $D$ is  sub-atomic, $uy$ is atomic in $D$. Then it follows from  \cite[Lemma~1.1]{AAZ92} that $uy$ is atomic in $D_S$. So $y$ is also atomic in $D_S$. Therefore, $D_S$ is a sub-atomic domain.\smallskip

 Conversely, assume that $D_S$ is a sub-atomic domain. Let $x$ be a nonunit atomic element of $D$, say $x=x_1\cdots x_n$, where each $x_i\in D$ is irreducible.  Let $y$ be a nonunit divisor of $x$ in $D$, say $x=yz$ for some $z\in D$.  We may assume that no nonunit element of $S$ divides $y$ in $D$. Indeed, we may write $y=y's$ for some $y'\in D$ and $s\in S$ such that $y'D \cap tD = y'tD$ for every $t\in S$. Since $s$ is a finite product of primes in $D$ or $s\in U(D)$, $s$ is atomic in $D$, and then we need only show that $y'$ is atomic in $D$. Note that no nonunit element of $S$ divides $y'$ in $D$. Thus, $x\notin S$, and hence $x$ can be written as $x = ar$ for some $a \in D\setminus U(D)$ and $r \in S$ such that no prime $p\in S$ divides $a$ in $D$.  Thus, $a=x_{i_1}\ldots x_{i_m}$, where $m\leq n$ and $x_{i_j}\in \{x_1,\ldots, x_n\}$. Since each $x_{i_j}\notin S$, $j=1,\ldots,m$, \cite[Lemma~1.1]{AAZ92} implies that each $x_{i_j}$ is irreducible in $D_S$. Thus, $a$ is atomic in $D_S$. Then it follows that $x$ is also atomic in $D_S$. Since $D_S$ is sub-atomic and $y\notin U(D_S)$ is a divisor of $x$ in $D_S$, $y$ is atomic in $D_S$. Then,  by an argument similar to the one above (the direct implication), one can easily check that $y$ is atomic in $D$. Consequently, $D$ is a sub-atomic domain.
\end{proof}

\section{Ascent on PSP-domains and GL-domains}\label{sec4}
This section is devoted to investigating the ascent of the sub-atomic and RIDF properties to polynomial rings. \smallskip

It is well-known that certain factorization properties, such as atomicity or the IDF-property, may or may not be preserved when passing from a domain $D$ to the polynomial ring $D[X]$. Understanding the behavior of the sub-atomic (resp., RIDF) property under such extensions is therefore a natural problem to consider. We begin this section with the following result, which states that the sub-atomic property ascends in the class of GL-domains when the domain $D$ contains no atoms. Before doing so, recall that a PSP-domain is a GL-domain, and a GL-domain is an AP-domain \cite[Proposition 3.2]{AZ07}. Then, by Proposition~\ref{NewPro},  an AP-domain (resp., a PSP-domain, a GL-domain)  is  sub-atomic.

\begin{Pro} \label{L1}
	Let $D$ be an antimatter domain. The following conditions are equivalent.
	\begin{enumerate}
		\item $D$ is a GL-domain.
		\smallskip
		\item $D[X]$ is a sub-atomic domain. 
	\end{enumerate}
\end{Pro}
\begin{proof} First note that, since $D$ is antimatter, no atomic element of $D[X]$ can lie in $D\setminus U(D)$. 

(1)$\Rightarrow $(2)  Assume that $D$ is a GL-domain.  Let $f\in D[X]\setminus D$ be an atomic element of $D[X]$. Then $f$ is a finite product of primitive (irreducible) elements of $D[X]$, and hence it  is primitive  since  $D$ is a GL-domain.  Hence, it follows that every nonunit divisor of $f$ in $D[X]$ is also primitive. Thus, by the degree argument, every nonunit divisor of $f$ is atomic  in $D[X]$. Therefore, $D[X]$ is a sub-atomic domain.
\smallskip

(2)$\Rightarrow $(1) Assume that $D[X]$ is a sub-atomic domain. Let $f$ and $g$ be two  primitive polynomials in $D[X]$. We may assume that $f$ and $g$ are non-constants. Since $f$ (resp., $g$) is primitive, it has an atomic factorization in $D[X]$. Thus, $h:=fg$ is atomic in $D[X]$. Now, suppose that $h=ah_1$ for some $a\in D$ and $h_1\in D[X]$. If $a\notin U(D)$,  since $a$ divides $h$ and $D[X]$ is a sub-atomic domain, then $a$ must have an irreducible divisor in $D[X]$. But  $D$ is divisor-closed in  $D[X]$, so $a$ has an irreducible divisor in $D$. This contradicts the fact that $D$ is an antimatter domain.  Thus $a$ is a unit of $D$, and therefore $h$ is primitive  in $D[X]$. Consequently, $D$ is a GL-domain.   
\end{proof}

As previously noted, every GL-domain is sub-atomic. However, there exist antimatter (hence sub-atomic) domains that are not GL-domains; see, for instance, \cite[Example~2.10]{AS75}. Consequently, in light of Proposition~\ref{L1}, the sub-atomic property does not ascend to polynomial rings in general.

\begin{Ep} \label{EP3}
 Let $K$ be an algebraically closed (or a real closed) field, and let $G$ be a subgroup of the additive group $\Q$ which is not  cyclic. Consider the group algebra $R:=K[G]$. Then, by \cite[Corollary 3]{BEGraz}, $R$ is an antimatter domain. On the other hand,   $R$ is a GL-domain since it is a GCD-domain \cite[Theorem~5.2]{GP74}. Therefore,  it follows from Proposition~\ref{L1} that the polynomial ring $D:=R[T]$ is  a sub-atomic domain which is not antimatter. Note that $D$ is an IDF-domain, and hence an RIDF-domain. However, it is clear that $D$ is not an atomic domain.
\end{Ep}

If $D$ is not an antimatter domain, the implication (2)$\Rightarrow $(1)  in Proposition~\ref{L1} may fail. 
\begin{Ep}
Let $K$ be a field and let $X$ and $Y$ be algebraically independent variables over $K$. Consider the numerical semigroup algebra $D := K[Y^2, Y^3] \subset K[Y]$. We claim that $D[X]$ is a sub-atomic domain, but $D$ is not a GL-domain.

To see that $D[X]$ is sub-atomic, observe that $D$ is a finitely generated $K$-algebra, hence Noetherian. Therefore, $D[X]$ is also Noetherian, and hence atomic. Consequently, $D[X]$ is a sub-atomic domain. On the other hand,  $f:= (Y^2X- Y^3) (Y^2X+ Y^3)\in D[X]$, which is a product of two primitive polynomials in $D[X]$, is not primitive since $f=Y^4(X^2-Y^2)$. Then, it follows that $D$ is not a GL-domain.
\end{Ep}

It is well-known that the properties of being antimatter and being MCD-finite are  independent. In particular, Eftekhari and Khorsandi proved the existence of integral domains that are antimatter yet fail to be MCD-finite. More precisely, in \cite[Theorem~2.5]{EK}, employing a modified variant of a powerful construction originally due to Roitman \cite{Rot}, the authors show that for any integral domain $D$ and any unit-closed subset $S \subseteq D$, one can construct a domain $R$ that is not MCD-finite and whose set of atoms is precisely $S$. By taking $S=\emptyset$, one obtains a domain without atoms, i.e., an antimatter domain. Moreover, by \cite[Theorem 2.1]{EK}, the polynomial ring $R[X]$ is not an IDF-domain. 

We now present an example demonstrating that the properties of being a GL-domain, being an antimatter domain, and being non-MCD-finite can all occur simultaneously in a single integral domain. We then apply  Proposition~\ref{L1} to get a sub-atomic domain which is not an IDF-domain.  
\begin{Ep}{\cite[Section 4]{GoPa}}\label{Ep4}

Let $\F_2$ denote the finite field with two elements.  Let
$
\mathcal{X}=\{X,Y,Z,T_1,T_2,T_3,\dots\}
$ be an infinite set of algebraically independent indeterminates over $\F_2$. Consider the integral domain
$$
    R_0 := \F_2 \bigg[ \big\{ X^\alpha, Y^\alpha, Z^\alpha : \alpha \in \Q_{\ge 0} \big\} \cup \Big \{\prod_{i=1}^{\infty} T_i^{\alpha_i} : \alpha_i \in \mathbb{Q}_{\ge 0} \ \forall \, i \in \mathbb{N} \Big\} \bigg].
$$
For the set of monomials
\[
  J := \bigg\{ \prod_{i=1}^{\infty} T_i^{\alpha_i} : \alpha_i \in \Q_0^+ \ \forall \, i \in \mathbb{N} \bigg\},
\]  
consider the following subring of $R_0$:
$$
    R :=\F_2\big[\big\{X^\alpha, Y^\alpha, Z^\alpha : \alpha \in \Q_{\ge 0} \big\} \cup 
     \big\{X^{\alpha}T, Y^\alpha T, Z^\alpha T : (\alpha,T) \in \Q_{>0} \times J \big\} \big].
$$
It is clear that $R$ is an antimatter domain, and by \cite{GoPa} it is a GL-domain which is not MCD-finite. Moreover, \cite[Theorem 2.1]{EK} implies that $D:=R[x]$, the polynomial ring over $R$ in the indeterminate $x$, is not an IDF-domain. Also, observe that $D$ is not atomic. On the other hand, it follows from Proposition~\ref{L1} that the polynomial ring $D$ is a sub-atomic domain. 
\end{Ep}

Recall that every GL-domain is a sub-atomic domain. In \cite{DG}, the authors show that the U-FF property does not
 ascend to polynomial extensions. Next we prove that the sub-atomic (resp., RIDF, U-FF) property ascends to polynomial rings in the class of PSP-domains.  We need the following lemma. 
 \begin{Lm}\label{LmNew}
Let $D$ be an integral domain and let $S\subseteq D$ be nonempty.
If $m$ is a maximal common divisor of $S$ and
$m=u p_1\cdots p_n$,
where $u$ is a unit and each $p_i$ is prime, then every common divisor
of $S$ divides $m$.
\end{Lm}

\begin{proof}
We proceed by induction on $n$.

If $n=0$, then $m$ is a unit. If $d$ is a common divisor of $S$, then $m\mid d$. By maximality of $m$, we get $d\sim m$, so $d\mid m$.

Now assume $n\ge 1$. Set
$p=p_n$ and 
$m'=u p_1\cdots p_{n-1}$,
so $m=p m'$. Since $m\mid s$ for every $s\in S$, we have $p\mid s$. Let 
$S'=\{s/p:s\in S\}$.
Then $m'$ is a common divisor of $S'$. We claim that $m'$ is a maximal common divisor of $S'$.
Let $c$ be a common divisor of $S'$ with $m'\mid c$.
Then $pc$ is a common divisor of $S$, and
$m\mid pc$.
Since $m$ is an MCD of $S$, we have $pc\sim m$, and so $c\sim m'$.
Hence $m'$ is maximal in $S'$.

By the induction hypothesis, every common divisor of $S'$ divides $m'$. Now, let $d$ be any common divisor of $S$.
For every $t=s/p\in S'$, we have $d\mid pt$. If $p\mid d$, write $d=pe$. Then $pe\mid pt$, so $e\mid t$.
Thus $e$ is a common divisor of $S'$, hence $e\mid m'$.
Therefore
$pe\mid p m'$, that is, $d\mid m$. Otherwise, $p\nmid d$, and then from the fact that $p$ is prime and  $d\mid pt$
we get $p\mid q$ whenever $pt=dq$. Write $q=pr$. Then
$pt=dpr$,
and so $t=dr$. Hence $d\mid t$.
Thus $d$ is a common divisor of $S'$, so $d\mid m'$, and therefore
$d\mid m$. In both cases every common divisor of $S$ divides $m$.
\end{proof}

\begin{Th}\label{Uff}
Let $D$ be an integral domain. Then the following statements hold:
\begin{enumerate}
    \item If $D$ is a GL-domain, then $D[X]$ is a sub-atomic domain.\smallskip

    \item If $D$ is a PSP-domain and an RIDF-domain, then $D[X]$ is an RIDF-domain.\smallskip

    \item If $D$ is a PSP-domain and a U-FFD, then $D[X]$ is a U-FFD.
\end{enumerate}
\end{Th}

\begin{proof}
Assume that $D$ is a GL-domain, and let $f \in D[X] \setminus D$ be an atomic element of $D[X]$. First, we prove the following two claims.\smallskip

\noindent \textsc{Claim 1.} The set of coefficients of $f$ has a unique MCD $c$ up to associates. Moreover, $c$ has a prime factorization in $D$.\smallskip

\noindent \textsc{Proof of Claim 1.}
If $f$ is primitive, then $1$ is the unique MCD of its coefficients up to associates (with the empty prime factorization), and we are done. Otherwise, since $f$ is atomic in $D[X]$, there exist $n, t \in \Z_+$, $A := \{a_1, \ldots, a_n\} \subseteq \mathcal{A}(D)$, and $F := \{f_1, \ldots, f_t\} \subseteq \mathcal{A}(D[X]) \setminus D$ such that
$$
f = \left(\prod_{i=1}^n a_i\right) \left(\prod_{i=1}^t f_i\right).
$$
Since $f \notin D$, we have $F \neq \emptyset$. On the other hand, since each $f_i \in \mathcal{A}(D[X]) \setminus D$ is primitive in $D[X]$ and $D$ is a GL-domain, the product $h := \prod_{i=1}^t f_i$ is primitive in $D[X]$. Because $f$ is not primitive, it follows that $A \neq \emptyset$. Set $c := \prod_{i=1}^n a_i$. Since $f = ch$ and $h$ is primitive, $c$ is an MCD of the coefficients of $f$.

To show uniqueness up to associates, let $c' \in D$ be another MCD of the coefficients of $f$. Then there exists a primitive polynomial $h' \in D[X]$ such that $f = c'h'$. Since $D$ is a GL-domain, it is an AP-domain, so each atom $a_i$ is prime in $D$, and hence also prime in $D[X]$ for every $i = 1, \ldots, n$. From the equality $ch = c'h'$ and the fact that $h'$ is primitive, each prime element $a_i$ divides $c'$ in $D$ ($i = 1, \ldots, n$). Consequently, $c$ divides $c'$ in $D$. Since $c$ is an MCD of the coefficients of $f$, the elements $c$ and $c'$ are associates in $D$. This proves the claim.\smallskip

\noindent \textsc{Claim 2.} Every divisor $d \in D \setminus U(D)$ of $f$ has a prime factorization in $D[X]$.\smallskip

\noindent \textsc{Proof of Claim 2.}
Let $d \in D \setminus U(D)$ be such that $f = df'$ for some $f' \in D[X]$. Then $d$ is a common divisor of the set of coefficients of $f$ in $D$. By Claim~1, this set admits an MCD $c$ that has a prime factorization in $D$. It then follows from Lemma~\ref{LmNew} that $d$ divides $c$ in $D$. Since $c$ has a prime factorization in $D$, it follows from Proposition~\ref{NewPro} that $d$ also has a prime factorization in $D$, and hence in $D[X]$. The claim is thus established.\smallskip

At this point, we have the tools needed to prove the three statements.\smallskip

$(1)$ Assume that $D$ is a GL-domain, and let $f$ be an atomic element of $D[X]$. Since $D$ is a GL-domain, it is sub-atomic. As $D$ is a divisor-closed subring of $D[X]$, we may assume that $f \in D[X] \setminus D$. Let $g_1 \in D[X] \setminus U(D)$ be a nonunit divisor of $f$, so $f = g_1 g_2$ for some $g_2 \in D[X]$. By Claim~1, we can write $f = ch$, where $h \in D[X]$ is primitive and $c$ has a prime factorization in $D$. Thus, $ch = g_1 g_2$. Canceling the prime divisors of $c$ in $D[X]$ yields $uh = g_1' g_2'$, where $u \in U(D)$, $c = c_1 c_2$ with $c_1, c_2 \in D$, and $g_i = c_i g_i'$ for $i = 1, 2$.

Since $h$ is primitive, both $g_1'$ and $g_2'$ are primitive in $D[X]$. If $c_1 \notin U(D)$, then $c_1$ divides $f$ in $D[X]$, so Claim~2 implies that $c_1$ has a prime  factorization in $D[X]$. Furthermore, if $\deg(g_1') = 0$, then $g_1' \in U(D)$ (as $g_1'$ is primitive), which forces $c_1 \notin U(D)$ because $g_1 = c_1 g_1' \notin U(D)$, and hence $g_1$ is atomic in $D[X]$. If $\deg(g_1') \ge 1$, then every nonunit divisor of $g_1'$ in $D[X]$ has positive degree (again because $g_1'$ is primitive), so a standard degree induction shows that $g_1'$ is a finite product of irreducible elements in $D[X]$. Therefore, $g_1 = c_1 g_1'$ is atomic in $D[X]$, proving that $D[X]$ is a sub-atomic domain.\smallskip

$(2)$ Assume that $D$ is a PSP-domain and an RIDF-domain, and let $f$ be an atomic element of $D[X]$. Since $D$ is an RIDF-domain and a divisor-closed subring of $D[X]$, we may assume that $f \in D[X] \setminus D$. Note that $D$ is also a GL-domain since it is a PSP-domain \cite[Proposition~3.2]{AZ07}. We must show that $f$ has only finitely many nonassociate irreducible divisors in $D[X]$.

Each non-constant irreducible divisor of $f$ in $D[X]$ is primitive, and hence super-primitive because $D$ is a PSP-domain; thus, \cite[Lemma~3.1(3)]{GZ22} implies that $f$ has only finitely many nonassociate non-constant irreducible divisors in $D[X]$. On the other hand, let $a_0 \in D$ be a constant irreducible (hence prime) divisor of $f$ in $D[X]$. By Claim~1 and Lemma~\ref{LmNew}, $a_0$ divides $c$ in $D$, so $a_0$ is an associate of one of the prime factors $a_1, \ldots, a_n$ of $c$ in $D$. Therefore, $f$ also has only finitely many nonassociate constant irreducible divisors in $D[X]$. Thus, $f$ satisfies the RIDF-property in $D[X]$, and consequently $D[X]$ is an RIDF-domain.\smallskip

$(3)$ Assume that $D$ is a PSP-domain and a U-FFD. Since every PSP-domain is a GL-domain \cite[Proposition~3.2]{AZ07}, it follows from $(1)$ that $D[X]$ is sub-atomic. Since $D$ is a U-FFD, it is an RIDF-domain by Proposition~\ref{P1dd}. Hence, by $(2)$, $D[X]$ is an RIDF-domain. Consequently, applying Proposition~\ref{P1dd} again implies that $D[X]$ is a U-FFD.
\end{proof}
 
 Next, we apply Theorem~\ref{Uff} to construct a class of RIDF-domains that are not IDF-domains. 
This construction further illustrates that the RIDF-property is strictly weaker than the IDF-property.

\begin{Ep} \label{gg}

Let $R$ be a PSP-domain which is not an IDF-domain. For instance, consider the Puiseux algebra   $R:=K[\Q_+]$, where $K=\Q(\sqrt[2^n]{q}:n\in \N)$ where $q>2$ is  prime.  Note that $R$ is a GCD-domain, and so it is a PSP-domain.  On the other hand, $R$ is not an IDF-domain by \cite[Example 3.8]{BE22}. \\
Now, consider the polynomial ring $D:=R[T]$. Then, by Theorem \ref{Uff}, $D$ is  an RIDF-domain. However, $D$ is not an IDF-domain since $R$ is not. Furthermore, observe that $D$ is not atomic, as $R$ contains elements (e.g.,  monomials and cyclotomic polynomials) that have no irreducible divisors; see \cite{BE22}. 
\end{Ep}

\smallskip

It is well known that every PSP-domain is a GL-domain \cite[Proposition 3.2]{AZ07}. Also, a PSP-domain is an 
MCD-finite domain \cite[Lemma 3.4]{BEJCA25}. This naturally raises the question of whether the RIDF-property (resp., U-FF property) ascends within these classes of domains.

 \begin{Qs} Does the RIDF-property (resp., U-FF property)  ascend in the classes of GL-domains and MCD-finite domains?
 \end{Qs}



\end{document}